\documentclass[a4paper,12pt]{amsart}
\usepackage{amsmath,amsfonts,amssymb,amsthm,latexsym,amscd,color}
\usepackage[dvips]{graphicx}
\usepackage[T1]{fontenc}

\usepackage{parskip}

\newcommand{\B}[1]{{\mathbf #1}}

\newcommand{\til}[1]{{\widetilde{#1}}}
\newtheorem{thm}{Theorem}[section]
\newtheorem{thm*}{Theorem}
\newtheorem{lem}[thm]{Lemma}
\newtheorem{prop}[thm]{Proposition}
\newtheorem{cor}[thm]{Corollary}

\newtheorem*{q*}{Question}

\theoremstyle{definition}

\newtheorem*{rem*}{Remark}
\newtheorem*{rems*}{Remarks}
\newtheorem*{cor*}{Corollary}

\def\B{\mathbf}

\newcommand\Diff{\operatorname{Diff}}

\newcommand\Ham{\operatorname{Ham}}

\newcommand\Aut{\operatorname{Aut}}

\newcommand\OP{\operatorname}

\def\A{\mathcal{A}}
\def\MA{\mathcal{A}_{\rm{Morse}}}

\def\Diff{\operatorname{Diff}}
\def\G{\mathcal{G}}

\def\Id{\operatorname{Id}}
\def\Im{\operatorname{Im}}

\def\X{\operatorname{X}}

\def\a{\alpha}
\def\b{\beta}

\def\g{\gamma}
\def\o{\omega}
\def\s{\sigma}

\begin{document}

\title[$L^p$-geometry of autonomous Hamiltonian diffeomorphisms]
{\textbf{On the $L^p$-geometry of autonomous Hamiltonian diffeomorphisms of surfaces.}}

\author[Brandenbursky]{Michael Brandenbursky}
\address{CRM, University of Montreal, Canada}
\email{michael.brandenbursky@mcgill.ca}
\author[Shelukhin]{Egor Shelukhin}
\address{CRM, University of Montreal, Canada \textit{and} Hebrew University of Jerusalem}
\email{egorshel@gmail.com}

\keywords{groups of Hamiltonian diffeomorphisms, braid groups, mapping class groups, quasi-morphisms, $L^p$-metrics}
\subjclass[2010]{Primary 53; Secondary 57}

\date{\today}
\maketitle

\begin{abstract}
We prove a number of results on the interrelation between the $L^p$-metric
on the group of Hamiltonian diffeomorphisms of surfaces and the subset $\A$
of autonomous Hamiltonian diffeomorphisms. More precisely, we show that
there are Hamiltonian diffeomorphisms of all surfaces of genus $g \geq 2$ or $g=0$
lying arbitrarily $L^p$-far from the subset $\A,$ answering a variant
of a question of Polterovich for the $L^p$-metric.
\end{abstract}


\section{Introduction and main results}


In contrast to the case of finite-dimensional Lie groups, the subset of
elements $\A$ of the group $\G$ of Hamiltonian diffeomorphisms lying
on a one-parameter subgroup, called autonomous Hamiltonian diffeomorphisms,
is very thin from several points of view. For example it is folklore in
symplectic geometry that, analogously to the case of general
diffeomorphism groups, $\A$ does not contain a neighborhood of the identity transformation
in, say, the $C^\infty$-topology.

L. Polterovich has proposed to study the inclusion $\A \subset \G$ from a metric point of view.
We study the geometry of this inclusion when $\G$ is endowed with the $L^p$-metric,
given by taking $L^p$-norms of Hamiltonian vector fields with respect to an auxiliary
Riemannian metric and volume form. The main result of this paper, applicable for surfaces of genus $g=0$ and 
exponent $p>2$, or $g \geq 2$ and exponent $p\geq 1$, is that with respect
to the $L^p$-metric, $\A$ is not coarsely dense: for every given number $C > 0$,
there exists an element $\phi \in \G$ whose $L^p$-distance to any element of $\A$ is greater than $C$.
Analogous results for the Hofer metric \cite{HoferMetric,LalondeMcDuffEnergy} on $\G$, given by taking the $C^0$-norms of normalized Hamiltonian functions, have recently been obtained for e.g. surfaces of genus $g \geq 2$ by Polterovich and the second named author \cite{PolterovichShelukhin}. Note that for exponents $1 \leq p \leq 2$, neither result follows from the other.

We further refine our main result showing that for every positive integer $k$,
the same conclusion holds when $\A$ is replaced by the subset $\A^k \subset \G$,
the image of $\A \times ... \times \A$ ($k$ times) in $\G$ under the multiplication map.
We remark that $\{\A^k\}_{k \in \B N}$ form an increasing sequence of subsets
of $\G$ whose union is the whole group $\G$, i.e. $\A$ is a generating subset for $\G$.

Our methods involve quasimorphisms on the group of Hamiltonian diffeomorphisms $\G$ of a surface that were introduced and first studied in \cite{GambaudoGhysCommutators} (and were further investigated in numerous other publications
\cite{BrandenburskyLpMetrics, BrandenburskyKedra1, BrandenburskyShelukhin, PolterovichDynamicsGroups, PySurfacesQm}).
The class of quasimorphisms on $\G$ that we use is produced by an averaging procedure from quasimorphisms on the fundamental groups of spaces upon which $\G$ acts. The spaces are configuration spaces $\OP{X}_n(\B\Sigma_g)$
of distinct ordered $n$-tuples of points on the orientable surface $\B\Sigma_g$.

These quasimorphisms were shown to be Lipschitz in the $L^p$-metric in many cases
\cite{BrandenburskyLpMetrics, BrandenburskyShelukhin} and are conjectured to have
this property for all $n$ and all surfaces $\B\Sigma_g$, the higher genus case
seeming more technically involved. In particular in the case of $\B T^2$ no such quasimorphism has been shown to be Lipschitz in the $L^p$-metric which shall be the subject of a subsequent work. Sometimes the above quasimorphisms have the additional property of vanishing on the subset $\A$ of autonomous Hamiltonian diffeomorphisms. Whenever both properties are satisfied, our result holds.

We therefore continue to show that for each surface $\B\Sigma_g$
other than $\B T^2$ there exists an integer $n > 0$ and an infinite-dimensional subspace
of homogeneous quasimorphisms on $\G$ transgressed from quasimorphisms on $\OP{X}_n(\B\Sigma_g)$
with the required two properties. The vanishing property follows from an analysis of the braids
traced by the  flow of an autonomous Morse Hamiltonian in the genus $g=0$ case, and by mapping
class group considerations in the higher genus case, the passage from Morse Hamiltonians to
general Hamiltonians being enabled by the Lipschitz property.

Whenever the vanishing property holds, we can show that the word metric with respect to
the generating set $\A$ has infinite diameter, which we hence prove along the way for
all genera other that $g=1$, thus reproving results from
\cite{Brandenbursky-surfaces, BrandenburskyKedra1, GambaudoGhysCommutators} in detail.

\subsection*{Acknowledgements}
We thank Leonid Polterovich for valuable conversations and for his comments on the manuscript.

Part of this work has been done during the first named author's stay in Mathematisches For\-schung\-sinstitut Oberwolfach and in Max Planck Institute for Mathematics in Bonn. He expresses his gratitude to both institutes. M.B. was supported by the Oberwolfach Leibniz fellowship and Max Planck Institute research grant.

Part of this work has been done during the second named author's stay in the Einstein Institute of Mathematics at the Hebrew University of Jerusalem. He expresses his gratitude to the institute and to Jake Solomon for his hospitality. E.S. was partially supported by ERC Starting Grant 337560.

Both authors were partially supported by the CRM-ISM fellowship. We thank CRM-ISM Montreal for the support and for a great research atmosphere.

\subsection{Preliminaries}

\subsubsection{Autonomous Hamiltonian diffeomorphisms}

Let $\B \Sigma_{g,k}$ be a compact connected orientable surface of genus
$g$ with $k$ boundary components equipped with a symplectic form $\o$,
and as usual we denote by $\G=\Ham(\B \Sigma_{g,k})$
the group of Hamiltonian diffeomorphisms (cf. \cite{BanyagaStructure,PolterovichBookGeometry}) of $\B \Sigma_{g,k}$. Let
$H\colon\B \Sigma_{g,k}\to\B R$
be a smooth function which vanishes in some neighborhood of $\partial\B\Sigma_{g,k}$.
It defines a time-independent vector field $X_H$ which is uniquely determined by the equation
$dH(v)=\o(v, X_H)$ for each vector field $v$. Let $h$ be the time-one map of the
flow $\{h_t\}$ generated by $X_H$. The diffeomorphism $h$ preserves $\o$ and every
diffeomorphism arising in this way is called {\em autonomous}. Since $X_H$ has
the property of being tangent to the level sets of $H$, each diffeomorphism $h_t$
preserves the level sets of~$H$.

We define the {\em autonomous norm} on the group $\G$ by
$$
\|f\|_{\OP{Aut}}:=
\min\left\{m\in \B N\,|\,f=h_1\cdots h_m
\text{ where each } h_i \text{ is autonomous}\right\}.
$$
The associated metric is defined by
$$
{\bf d}_{\OP{Aut}}(f,h):=\|fh^{-1}\|_{\OP{Aut}}.
$$
Since the set of autonomous diffeomorphisms is invariant under conjugation the
autonomous metric is bi-invariant. For the same reason the subgroup generated
by the autonomous diffeomorphisms is normal (and non-empty). Hence by a fundamental theorem of Banyaga \cite{BanyagaStructure} stating that
$\G$ is simple for $k=0$ and the kernel of the Calabi homomorphism is simple for $k\neq 0$, it is easy to see that the set of autonomous diffeomorphisms generates $\G$. Therefore the autonomous norm of any element $f\in\G$ is well-defined.
We note that since the autonomous norm is subadditive, its {\em stabilization}
$$\|f\|_{\OP{st}}: = \lim_{k \to \infty}\frac{1}{k}\|f^k\|_{\OP{Aut}}$$
is well-defined.

\subsubsection{The $L^p$-metric}
Let $\B M$ denote a compact connected oriented Riemannian manifold (possibly with boundary) with a volume form $\mu$.
We denote by $\G=\Diff_{c,0}(\B M,\mu)$ the identity component of the group of diffeomorphisms of $\B M$
preserving $\mu$, that are identity near the boundary if $\partial\B M\neq 0$.
Alternatively, in the case of non-empty boundary, one can consider the open manifold
$\B M\setminus\partial\B M$ and take compactly supported diffeomorphisms preserving $\mu$.

Given a path $\{f_t\}$ in $\G$ between $f_0$ and $f_1$, we define its $L^p$-length by
\[l_p(\{f_t\}) = \int_0^1 dt \, (\int_{\B M} |X_t|^p \mu)^{\frac{1}{p}},\]
where $X_t = \frac{d}{dt'}|_{t'=t} f_{t'} \circ f_t^{-1}$ is the time-dependent
vector field generating the path $\{f_t\}$, and $|X_t|$ its length with respect
to the Riemannian structure on $\B M$. As is easily seen by a displacement argument,
this length functional determines a non-degenerate metric on $\G$ by the formula
\[\B d_p(f_0,f_1) = \inf \; l_p(\{f_t \}),\]
where the infimum runs over all paths $\{f_t\}$ in
$\G$ between $f_0$ and $f_1$. It is immediate that
this metric is right-invariant. We denote the
corresponding norm on the group by
\[||f||_p = \B d_p(\Id,f).\]
Clearly $\B d_p(f_0,f_1) = ||f_1f_0^{-1}||_p$. Similarly one has the $L^p$-norm on the universal
cover $\til{\G}$ of $\G$, defined for $\til{f} \in \til{\G}$ as
\[||\til{f}||_p = \inf \; l_p(\{f_t\}),\]
where the infimum is taken over all paths $\{f_t\}$ in the class of $\til{f}$.
For more information see \cite{ArnoldKhesin}.

We note that up to bi-Lipschitz equivalence of metrics ($d$ and $d'$ are equivalent
if $\frac{1}{C}d \leq d' \leq C d$ for a certain constant $C>0$) the $L^p$-metrics
on $\G$ and on $\til{\G}$ are independent of the choice Riemannian structure and of
the volume form $\mu$ compatible with the orientation on $M$. In particular questions on the large-scale geometry of  the $L^p$-metric enjoy the same invariance property.

\subsubsection{Quasimorphisms}
The notion of a quasimorphism will play a key role in our arguments.
Quasimorphisms are helpful tools for the study of non-abelian groups,
especially those that admit few or no homomorphisms to the reals.
A quasimorphism $\phi\colon \Gamma\to\B R$ on a group $\Gamma$ is a real-valued function that satisfies for any two $\gamma_1,\gamma_2 \in \Gamma$ the relation
\[\phi(\gamma_1\gamma_2) = \phi(\gamma_1) + \phi(\gamma_2) + b(\gamma_1,\gamma_2),\]
for a function $b\colon \Gamma\times \Gamma\to\B R$ that is uniformly bounded:
\[D_\phi: = \sup_{\Gamma\times \Gamma} |b| < \infty.\]

A quasimorphism $\overline{\phi}\colon\Gamma\to\B R$ is called \textit{homogeneous} if
$\overline{\phi}(f^k) = k \overline{\phi}(f)$ for all $f\in\Gamma$ and
$k\in\B Z$. To any quasimorphism $\phi\colon\Gamma\to\B R$ there corresponds a unique
homogeneous quasimorphism $\overline{\phi}$ that differs from $\phi$ by a bounded function:
\[\sup_{\Gamma} |\overline{\phi} - \phi| <\infty.\]
It is called the \textit{homogenization} of $\phi$ and satisfies
\[\overline{\phi}(f) = \lim_{k\to\infty}\frac{\phi(f^k)}{k}.\]
The key property of a homogeneous quasimorphism is that it restricts to an actual homomorphism on every abelian subgroup.
A homogeneous quasimorphism $\overline{\phi}\colon\Gamma\to\B R$ is called \textit{genuine} if it is not a homomorphism. We refer to \cite{CalegariScl} for more information about quasimorphisms.

\subsection{Main results}

Let $\G=\Ham(\B\Sigma_g)$ be a group of Hamiltonian diffeomorphisms of $\B\Sigma_g$.
The main technical result of this paper is the following

\begin{thm*}\label{T:quasimorphisms-two-prop}
Let $g=0$ and $p>2$, or $g>1$ and $p\geq 1$. Then there exists an infinite-dimensional space of
homogeneous quasimorphisms $\G\to\B R$ that are both Lipschitz in
the $L^p$-metric and vanish on all autonomous Hamiltonian diffeomorphisms.
\end{thm*}

\begin{rem*}\label{R:BK-work}
Let $\B D^2$ denote an open unit disc in the Euclidean plane, i.e $\B D^2:=\B\Sigma_{0,1}$.
Then the above theorem follows from Theorem 2.3 in \cite{BrandenburskyKedra1} combined
with Theorem 2 in \cite{BrandenburskyLpMetrics}.  If $g\neq 1$ then the statement of
the above theorem holds in the case of $\Ham(\B\Sigma_{g,k})$ as well.
The proof of this fact follows immediately from the proof of Theorem \ref{T:quasimorphisms-two-prop}.
\end{rem*}

As a corollary we obtain the main result of this paper,
showing that $\A\subset (\G,\B d_p)$ is not coarsely dense.
For a diffeomorphism $f\in\G$ and a subset $\mathcal{S}\in\G$
we define the distance $\B d_p(\phi,\mathcal{S})$ from $f$ to $\mathcal{S}$ by
\[\B d_p(f,\mathcal{S}):= \inf_{h\in\mathcal{S}}\B d_p(f, h).\]

\begin{cor}\label{C:Aut-thin-infinity-1}
For every $K\geq 0$ there exists a Hamiltonian diffeomorphism $f'\in\G$
such that $\B d_p(f',\A)\geq K$.
\end{cor}

\begin{proof}
Let $\overline{\phi}$ be a non-vanishing homogeneous quasimorphism
provided by Theorem \ref{T:quasimorphisms-two-prop}.
Hence there exists $f\in\G$ such that $\overline{\phi}(f)\neq 0$.
Then for $h\in\A$ we have the following inequalities:
\[|\overline{\phi}(f)| - D_{\overline{\phi}} \leq |\overline{\phi}(f)-\overline{\phi}(h^{-1})+b(f,h^{-1})|=
|\overline{\phi}(fh^{-1})| \leq C_{\overline{\phi}}\cdot \B d_p(f,h),\]
where the rightmost inequality follows from the Lipschitz property, and the leftmost inequality follows from the vanishing condition, since $h^{-1}\in\A.$ We conclude that
\[|\overline{\phi}(f)| - D_{\overline{\phi}}\leq C_{\overline{\phi}}\cdot\B d_p(f,\A).\]
Therefore denoting $c:=|\overline{\phi}(f)| > 0$ and taking $f' =f^m,$ we observe that $\B d_p(f',\A)$ satisfies
\[c\cdot m - D_{\overline{\phi}}\leq C_{\overline{\phi}}\cdot\B d_p(f',\A),\]
and the proof follows.
\end{proof}

Slightly upgrading the proof, we have
\begin{cor}\label{C:Aut-thin-infinity-2}
For every $K\geq 0$ and $k \in \B N$ there exists a Hamiltonian diffeomorphism
$f'\in\G$ such that $\B d_p(f',\A^k) \geq K$.
\end{cor}

Indeed for each $h = h_1 \circ\ldots\circ h_k \in\A^k$ we have
$|\overline{\phi}(h)| \leq (k-1)D_{\overline{\phi}}$.
Therefore $\B d_p(f^m, \A^k)$ satisfies
\[|\overline{\phi}(f)|\cdot m-k D_{\overline{\phi}}\leq C_{\overline{\phi}}\cdot\B d_p(f^m, \A^k).\]
By taking a sufficiently large $m$ and $f'=f^m$, we conclude the proof of the corollary.

\begin{rem*}
Note that every homogeneous quasimorphism provided by Theorem \ref{T:quasimorphisms-two-prop} is genuine. This fact together with Corollary \ref{C:Aut-thin-infinity-2} imply that the metric group $(\G,\B d_{\Aut})$ has an infinite diameter. Moreover, the same fact implies that $\G$ is stably unbounded with respect to the autonomous metric, i.e. there exists $g\in\G$ such that $\|g\|_{\OP{st}}>0$.
\end{rem*}

\section{Proof of the main technical result}
Our proofs of the genus zero case and of the case of hyperbolic surfaces are different.
Before we start proving the main result, let us recall two constructions, one due to
Gambaudo and Ghys \cite{GambaudoGhysCommutators} and the other due to Polterovich \cite{PolterovichDynamicsGroups}, of quasimorphisms on the group $\G$ of Hamiltonian diffeomorphisms of compact surfaces.

\subsection{Quasimorphism constructions}
Let $g=0$. In what follows we recall a construction due to
Gambaudo and Ghys \cite{GambaudoGhysCommutators}, cf. \cite{BrandenburskyShelukhin},
 of a homogeneous quasimorphism on the group $\G$ of Hamiltonian diffeomorphisms of the two-sphere $\B S^2$
which is produced from a quasimorphism on the spherical pure braid group $\B P_n(\B S^2)$.

\subsubsection{Gambaudo-Ghys construction}
Let $\{f_t\}\in\G$ be an isotopy from the identity to $f\in\G$
and let $w\in \B S^2$ be a basepoint. For each $x\in\B S^2$ let us
choose a short geodesic from $w$ to $x$ and denote it by $s_{wx}$.
For $y\in\B S^2$ we define a loop $\gamma_{y,w}\colon [0,1]\to\B S^2$
to be a concatenation of paths $s_{wy}$, $f_{3t-1}(y)$
(here $t\in [\frac{1}{3},\frac{2}{3}]$) and $s_{g(y)w}$.

Let $\X_n(\B S^2)$ be the configuration space of all ordered $n$-tuples
of pairwise distinct points in the sphere $\B S^2$. It's fundamental group
$\pi_1(\X_n(\B S^2))$ is identified with the spherical pure braid group $\B P_n(\B S^2)$.
Fix a basepoint $z=(z_1,\ldots,z_n)$ in $\X_n(\B S^2).$
For almost every $x=(x_1,\ldots,x_n)\in\X_n(\B S^2)$ the
$n$-tuple of loops $(\gamma_{x_1,z_1},\ldots,\gamma_{x_n,z_n})$ is
a based loop in the configuration space $\X_n(\B S^2)$.
Let
$$
\gamma(f_t,x)\in \B P_n(\B S^2)=\pi_1(\X_n(\B S^2),z)
$$
be an element represented by this loop.

Let $\overline{\phi}\colon \B P_n(\B S^2)\to \B R$ be a homogeneous quasimorphism.
Define the quasimorphism
$\Phi_n\colon \ \widetilde{\G}\to \B R$ by
\begin{equation*}\label{eq:GG-not-homogeneous}
\Phi_n(\{f_t\}):=\int\limits_{\X_n(\B S^2)}\overline{\phi}(\g(f_t;x))d{x}
\end{equation*}

The fact that the above function is a well defined quasimorphism follows from \cite{BrandenburskyShelukhin}.
Recall that Smale proved that $\pi_1(\G)=\B Z/\B{2Z}$, see \cite{Smale}. Hence the homogenization $\overline{\Phi}_n$ of $\Phi_n,$
being a homomorphism on all abelian subgroups (cf. \cite{CalegariScl}),
descends to a well defined homogeneous quasimorphism $\overline{\Phi}_n\colon \G\to\B R$
which neither depends on the choice of short geodesics, nor on the choice of the base point. For any choice of an isotopy $\{f_{t,k}\}$ between the identity and $f^k$ it can be computed as
\begin{equation*}\label{eq:GG-homogeneous}
\overline{\Phi}_n(f):=\lim_{k\to \infty}\Phi_n(\{f_{t,k}\})/k.
\end{equation*}

\subsubsection{Polterovich construction}
\label{ssec-Polterovich-construction}

Let $g>1$ and $z\in\B\Sigma_g$. Denote by $\Diff_0(\B\Sigma_g, \o)$ the identity component
of the group of area preserving diffeomorphisms of $\B\Sigma_g$ and by $\G$ its subgroup
of Hamiltonian diffeomorphisms. It is known that the group $\pi_1(\B\Sigma_g,z)$ admits
infinitely many linearly-independent homogeneous quasimorphisms, see \cite{EpsteinFujiwara}.
Let
$$
\overline{\phi}\colon\pi_1(\B\Sigma_g,z)\to\B R
$$
be a non-trivial homogeneous quasimorphism.
For each $x\in\B\Sigma_g$ let us choose an arbitrary short geodesic path from $x$ to $z$.
In \cite{PolterovichDynamicsGroups} L. Polterovich constructed the induced
\emph{non-trivial} homogeneous quasimorphism $\overline{\Phi}$ on $\Diff_0(\B\Sigma_g,\o)$ as follows.

For each $x\in\B\Sigma_g$ and an isotopy $\{f_t\}_{t\in[0,1]}$ between $\Id$
and $f$ let $f_x$ be the closed loop in $M$ which is the concatenation of a short geodesic
path from $z$ to $x$, the path $f_t(x)$ and a short geodesic path from $f(x)$ to $z$.
Denote by $[f_x]$ the corresponding element in $\pi_1(\B\Sigma_g,z)$ and set
$$
\Phi(f):=\int\limits_{\B\Sigma_g} \overline{\phi}([f_x])\o\qquad\qquad
\overline{\Phi}(f):=\lim\limits_{k\to\infty}\frac{1}{k}\int\limits_{\B\Sigma_g}\overline{\phi}([(f^k)_x])\o.
$$
The maps $\Phi$ and $\overline{\Phi}$ are well-defined quasimorphisms because the
center $Z(\pi_1(\B\Sigma_g,z))$ is trivial and every diffeomorphism in $\Diff_0(\B\Sigma_g, \o)$
is area-preserving. In addition, the quasimorphism $\overline{\Phi}$ depends neither on the choice
of the family of geodesic paths, nor on the choice of the base point $z$. Moreover, if $\overline{\phi}$
is not a homomorphism, then $\overline{\Phi}$ is not a homomorphism, i.e.
if $\overline{\phi}$ is genuine then $\overline{\Phi}$ is also genuine.
For more details see \cite{PolterovichDynamicsGroups}.

Recall that $\G$ is the commutator subgroup of $\Diff_0(\B\Sigma_g, \o)$ (cf. \cite{BanyagaStructure}).
It follows that every genuine homogeneous quasimorphism $\overline{\phi}$ on $\pi_1(\B\Sigma_g,z)$
defines a genuine homogeneous quasimorphism $\overline{\Phi}_{\G}$ on $\G$ which
is the restriction of $\overline{\Phi}$ to $\G$.

\subsection{Continuity of the Gambaudo-Ghys and Polterovich quasimorphisms}

The aim of this subsection is to prove the following technical results which
will be used in the proof of Theorem \ref{T:quasimorphisms-two-prop}.

\begin{thm}\label{T:Morse-Ham-sphere}
Let $H\colon\B S^2\to\B R$ and $\{H_k\}_{k=1}^\infty$ be a sequence of functions such that
each $H_k\colon\B S^2\to\B R$ and $H_k\xrightarrow[k\rightarrow\infty]{}H$ in
$C^1$-topology. Let $h_1$ and $h_{k,1}$ be the time-one maps of the Hamiltonian flows
generated by $H$ and $H_k$ respectively. Then for each $n$
$$
\lim\limits_{k\to\infty}\overline{\Phi}_n(h_{k,1})=\overline{\Phi}_n(h_1),
$$
where $\overline{\Phi}_n$ is a quasimorphism induced by the Gambaudo-Ghys construction.
\end{thm}

\begin{thm}\label{T:Morse-Ham-hyperbolic}
Let $g>1$, $H\colon\B\Sigma_g\to\B R$ and $\{H_k\}_{k=1}^\infty$ be a sequence of functions such that
each $H_k\colon\B\Sigma_g\to\B R$ and $H_k\xrightarrow[k\rightarrow\infty]{}H$ in
$C^1$-topology. Let $h_1$ and $h_{k,1}$ be the time-one maps of the Hamiltonian flows
generated by $H$ and $H_k$ respectively. Then
$$
\lim\limits_{k\to\infty}\overline{\Phi}(h_{k,1})=\overline{\Phi}(h_1),
$$
where $\overline{\Phi}$ is any quasimorphism induced by the Polterovich construction.
\end{thm}

\begin{proof}
In \cite{BrandenburskyShelukhin} the authors proved the following

\begin{thm}[\cite{BrandenburskyShelukhin}]\label{T:lipschitz-sphere}
Let $n>0$ and $\overline{\Phi}_n$ be a homogeneous quasimorphism induced by the Gambaudo-Ghys construction.
Then $\overline{\Phi}_n$ is Lipschitz with respect to the $L^3$-metric on the group $\G$ of Hamiltonian
diffeomorphisms of $\B S^2$, i.e. there exists $C>0$ such that $\forall h\in\G$
$$
\overline{\Phi}_n(h)\leq C\|h\|_3.
$$
\end{thm}
In addition, in \cite{BrandenburskyLpMetrics} the first named author proved the following

\begin{thm}[\cite{BrandenburskyLpMetrics}]\label{T:lipschitz-hyperbolic}
Let $\B \Sigma_g$ be a closed hyperbolic surface, and $\overline{\Phi}$ a homogeneous
quasimorphism induced by the Polterovich construction. Then $\overline{\Phi}$ is Lipschitz
with respect to the $L^3$-metric on the group $\G$ of Hamiltonian diffeomorphisms of $\B\Sigma_g$,
i.e. there exists $C'>0$ such that $\forall h\in\G$
$$
\overline{\Phi}(h)\leq C'\|h\|_3.
$$
\end{thm}

\begin{lem}\label{lem:epsilon-close}
Let $g\neq 1$ and $H\colon\B\Sigma_g\to\B R$ be a smooth function.
Then for any $\epsilon>0$ and $p\in\B N$ there exists $\delta_p>0$, such that if
$H$ is $\delta_p$-close to a smooth function $F\colon\B\Sigma_g\to\B R$ in $C^1$-topology, then
$$
\B d_3(h^p,f^p)<\epsilon,
$$
where $h_t$ and $f_t$ are the Hamiltonian flows generated by $H$ and $F$ respectively,
and $h$ and $f$ are time-one maps of these flows.
\end{lem}
\begin{proof}
We replace $\B D^2$ by $\B\Sigma_g$ and $\B d_2$ with $\B d_3$ in the proof of Lemma 3.7 in \cite{BrandenburskyKedra1}.
Now the proof is identical to the proof of Lemma 3.7 in \cite{BrandenburskyKedra1}.
\end{proof}

\begin{prop}\label{P:Morse-Ham-epsilon}
Let $g\neq 1$ and $H\colon\B\Sigma_g\to\B R$. Then for any $\epsilon>0$ there exists $\delta>0$,
such that if $F\colon\B\Sigma_g\to\B R$ is $\delta$-close to $H$ in $C^1$-topology then:
$$
\left|\overline{\Psi}(h)-\overline{\Psi}(f)\right|\leq\epsilon,
$$
where $\overline{\Psi}=\overline{\Phi}_n$ in case when $\B\Sigma_g=\B S^2$, and
$\overline{\Psi}=\overline{\Phi}$ in the higher genus case, and $h$ and $f$
are time-one maps of flows generated by $H$ and $F$.
\end{prop}

\begin{proof}
Fix some $\epsilon>0$. Denote by $K$ the constant which was defined in
Theorem \ref{T:lipschitz-sphere} in case of genus zero
(it was denoted by $C$), and in Theorem \ref{T:lipschitz-hyperbolic}
in the higher genus case (it was denoted by $C'$).
Take $p\in\B N$ such that $\frac{D_{\overline{\Psi}}+K}{p}<\epsilon$. It
follows from Lemma \ref{lem:epsilon-close} that there exists $\delta_p>0$, such
that if $F$ is $\delta_p$-close to $H$ in $C^1$-topology, then
$\B d_3(f^p,h^p)<1$. Thus we obtain
$$
\left|\overline{\Psi}(f)-\overline{\Psi}(h)\right|=
\frac{1}{p}\left|\overline{\Psi}(f^p)-\overline{\Psi}(h^p)\right|\leq
\frac{D_{\overline{\Psi}}+\left|\overline{\Psi}(f^ph^{-p})\right|}{p}.
$$
Depending on the case, it follows from Theorem \ref{T:lipschitz-sphere}
or from Theorem \ref{T:lipschitz-hyperbolic} that
$$
\left|\overline{\Psi}(f^ph^{-p})\right|\leq K\B d_3(\Id,f^ph^{-p})=K\B d_3(f^p,h^p)<K.
$$
Thus
$$
\left|\overline{\Psi}(f)-\overline{\Psi}(h)\right|<\frac{D_{\overline{\Psi}}+K}{p}<\epsilon.
$$
\end{proof}
Proposition \ref{P:Morse-Ham-epsilon} concludes the proof of
Theorems \ref{T:Morse-Ham-sphere} and \ref{T:Morse-Ham-hyperbolic}.
\end{proof}

\subsection{Proof of Theorem \ref{T:quasimorphisms-two-prop} - the spherical case}

Let $\B B_n$ and $\B B_n(\B S^2)$ be the standard Artin braid group and
the spherical braid group on $n$ strands respectively. The group $\B B_n$ admits the following presentation:
$$
\B B_n=\langle\s_1,\ldots,\s_{n-1}|\hspace{2mm} \s_i\s_j=\s_j\s_i,\hspace{2mm}|i-j|\geq2;\hspace{2mm}\s_i\s_{i+1}\s_i=\s_{i+1}\s_i\s_{i+1},\rangle
$$
where $\s_i$ is the $i$-th Artin generator of $\B B_n$. The group $\B B_n(\B S^2)$ has
the same generators and relations as $\B B_n$ and one extra relation given by
$$
\delta_n:=\s_1\ldots\s_{n-2}\s_{n-1}^2\s_{n-2}\ldots\s_1=1.
$$
It follows that these relations define an epimorphism
$$
\Pi\colon \B B_n\to \B B_n(\B S^2).
$$
Note that if $n>3$ then both $\B B_n$ and $\B B_n(\B S^2)$ are infinite groups.
Fix $n>3$ and let $\eta_{i,n}:=\s_{i-1}\ldots\s_2\s_1^2\s_2\ldots\s_{i-1}\in \B B_n$,
be the braid presented in Figure \ref{fig:braids-eta-i-n}.
\begin{figure}[htb]
\centerline{\includegraphics[height=1.7in]{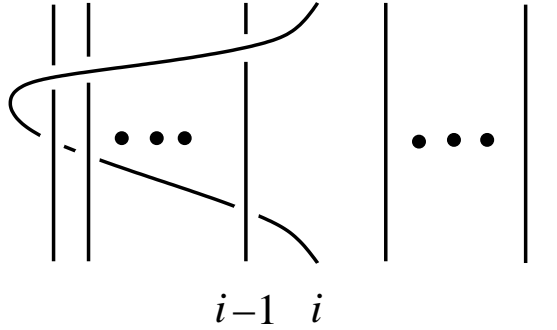}}
\caption{\label{fig:braids-eta-i-n} The braid $\eta_{i,n}$.}
\end{figure}

Denote by $\B A_n<\B B_n$ the free abelian group of rank $n-1$
generated by $\{\eta_{i,n}\}_{i=2}^n$,
and by $\B A_n(\B S^2)<\B B_n(\B S^2)$ the abelian group $\Pi(\B A_n)$.

For a group $\Gamma$ let us denote the space of homogeneous
quasimorphisms on $\Gamma$ by $Q(\Gamma)$. Since the mapping class group of
an $n$-punctured sphere is isomorphic to the quotient of $\B B_n(\B S^2)$,
and by results of Bestvina-Fujiwara \cite{BestvinaFujiwara} the space of
homogeneous quasimorphisms on the mapping class group of an $n$-punctured
sphere is infinite dimensional (recall that $n>3$), it follows that $\dim(Q(\B B_n(\B S^2)))=\infty$.
The group $\B A_n(\B S^2)$ is abelian of finite rank, hence
$$
\dim(Q(\B B_n(\B S^2), \B A_n(\B S^2)))=\infty,
$$
where $Q(\B B_n(\B S^2), \B A_n(\B S^2))$
is the subspace of $Q(\B B_n(\B S^2))$ which consists of homogeneous
quasimorphisms on $\B B_n(\B S^2)$ which vanish on $\B A_n(\B S^2)$.

Let $Q(\G,\A)$ be the space of homogeneous quasimorphisms on $\G$ which
vanish on the set $\A$ of autonomous diffeomorphisms. Also denote by
$$
\mathfrak{GG}_n\colon Q(\B B_n(\B S^2))\to Q(\G)
$$
the linear map which is given by the Gambaudo-Ghys construction.
By the result of Ishida \cite{Ishida} the map $\mathfrak{GG}_n$ is injective for all $n$.
Since for $n>3$ we have
$$
\dim(Q(\B B_n(\B S^2), \B A_n(\B S^2)))=\infty
$$
 and $\mathfrak{GG}_n$ is injective,
then in order to show that
$$\dim(Q(\G, \A))=\infty$$
it is enough to show that
$$
\Im(\mathfrak{GG}_n|_{Q(\B B_n(\B S^2), \B A_n(\B S^2))})\subset Q(\G,\A).
$$

Let $p>2$. In \cite{BrandenburskyShelukhin} the authors proved that every quasimorphism
which lies in the image of the Gambaudo-Ghys map $\mathfrak{GG}_n$
is Lipschitz with respect to the $L^p$-metric on $\G$.
It follows that in order to complete the proof, it is enough to prove the following.

\begin{prop}\label{P:GG-autonomous-injective}
Let $\overline{\phi}\colon \B B_n(\B S^2)\to\B R$ be a homogeneous quasimorphism
which vanishes on $\B A_n(\B S^2)$. Then the induced homogeneous
quasimorphism $\overline{\Phi}_n\colon\G\to\B R$ vanishes on the set $\A$.
\end{prop}

\begin{proof}
Since Morse functions on $\B S^2$ form a dense subset in the set of
all smooth functions in the $C^1$-topology \cite{Milnor}, by Theorem \ref{T:Morse-Ham-sphere}
it is enough to prove the statement for Morse autonomous diffeomorphisms.
We say that a Hamiltonian diffeomorphism $h$ is Morse autonomous
if it is generated by some Morse function $H\colon\B S^2\to\B R$.
The set of all Morse autonomous diffeomorphisms is denoted by $\MA$.

The idea of the proof relies on the fact that $n$ points on different level curves of $H$ trace a braid
which is "almost conjugate" to a braid in $\B A_n(\B S^2)$.
More precisely, let $h\in\MA$ which is generated by a function $H$ and take a
point $x=(x_1,\ldots,x_n)\in\X_n(\B S^2)$ such that each
$x_i$ lies on a different regular level set of $H$. Let $\{h_t\}_{t \in [0,\infty)}$ be the Hamiltonian isotopy generated by $H$ and put $h:=h_1$.
We have the identity of braids
$$
\g(h_k,x)=\a_{h,k,x}\b_1^{m_{1,h,k,x}}\ldots\b_{n-1}^{m_{n-1,h,k,x}}\a'_{h,k,x},
$$
where the word length of the braids $\a_{h,k,x}$ and $\a'_{h,k,x}$ is
universally bounded by some natural number $C$ which depends only on $n$,
all the braids $\b_i$ commute with each other and each $\b_i$ is conjugate
in $\B B_n(\B S^2)$ to some $\Pi(\eta_{j,n})\in \B A_n(\B S^2)$. Note that a similar identity of braids in the case of a disc was established in \cite[Theorem 4.5]{BrandenburskyKnots}, cf. \cite{BrandenburskyThesis}.
It follows that
\begin{eqnarray*}
\lim_{k\to\infty}\frac{|\overline{\phi}(\g(h_k,x))|}{k}&\leq&
\lim_{k\to\infty}\frac{|\overline{\phi}(\a_{h,k,x})|+|\overline{\phi}(\a'_{h,k,x})|}{k}\hspace{1mm}\\
&+&\lim_{k\to\infty}\frac{\sum_{i=1}^{n-1}|m_{i,h,k,x}||\overline{\phi}(\b_i)|+2D_{\overline{\phi}}}{k}.
\end{eqnarray*}
Since $\overline{\phi}$ is a homogeneous quasimorphism, the value of $\overline{\phi}$ on conjugate elements is the same and by our hypothesis
$\overline{\phi}\in Q(\B B_n(\B S^2), \B A_n(\B S^2))$, that is
$\overline{\phi}(\Pi(\eta_{i,n}))=0$ for each $i$, we conclude that
$$
\lim_{k\to\infty}\frac{|\overline{\phi}(\g(h_k,x))|}{k}=0.
$$

Recall that $H$ is a Morse function, and so it has finitely many critical points.
Thus the complement in $\X_n(\B S^2)$ of the set of all the $n$-tuples of different points in $\B S^2$ which
lie on different regular level sets of $H$ is of measure zero. Note that since $h$ is an autonomous diffeomorphism we have $h^k=h_k$. Therefore
$$
\overline{\Phi}_n(h)=\int\limits_{\X_n(\B S^2)}\lim_{k\to\infty}\frac{|\overline{\phi}(\g(h^k,x))|}{k}\thinspace d{x}=
\int\limits_{\X_n(\B S^2)}\lim_{k\to\infty}\frac{|\overline{\phi}(\g(h_k,x))|}{k}\thinspace d{x}=0
$$
and the proof follows.
\end{proof}

\subsection{Proof of Theorem \ref{T:quasimorphisms-two-prop} - the hyperbolic case}

\subsubsection{Curves traced by Morse autonomous flows}
\label{sec-curves-aut-flows}
Let $\{h_t\}$ be an autonomous flow generated by a Morse function $H\colon\B\Sigma_g\to\B R$
and set $h:=h_1$. Let $x\in\B\Sigma_g$ which satisfies the following conditions:

\begin{itemize}
\item
$x$ is a regular point of $H$,
\item
$x$ belongs to only one connected component, i.e. a simple closed curve in $\B\Sigma_g$, of the set $H^{-1}(H(x))$.
\end{itemize}

Such a set of points in $\B\Sigma_g$ is denoted by $\mathrm{Reg}_H$.
Note that the measure of $\B\Sigma_g\setminus\mathrm{Reg}_H$ is zero.
For each $x\in\mathrm{Reg}_H $ let $c_x\colon[0,1]\to\B\Sigma_g$
be an injective path (on $(0,1)$), such that $c_x(0)=c_x(1)=x$ and
its image is a simple closed curve which is a connected component of $H^{-1}(H(x))$.
For every $y_1,y_2\in \B\Sigma_g$ choose an injective map $s_{y_1y_2}\colon[0,1]\to \B\Sigma_g$
whose image is a short geodesic path from $y_1$ to $y_2$. Define
\begin{equation}\label{eq:gamma-reg}
\gamma_x(t):=
\begin{cases}
s_{zx}(3t) &\text{ for } t\in \left [0,\frac13\right ]\\
c_x(3t-1) &\text{ for } t\in \left [\frac13,\frac23\right ]\\
s_{xz}(3t-2) & \text{ for } t\in \left [\frac23,1\right ]
\end{cases}.
\end{equation}
Denote by $[\gamma_x]$ the corresponding element in $\pi_1(\B\Sigma_g, z)$.
Let $x\in\mathrm{Reg}_H$ and let $[h_x]$ be an element in $\pi_1(\B\Sigma_g, z)$
represented by a path which is a concatenation of paths $s_{zx}$, $h_t(x)$ and $s_{h(x)z}$.
Then for each $k\in\B N$ the element $[h^k_x]$ can be written as a product
\begin{equation}\label{eq:braid-morse-aut}
[h^k_x]=\a'_{h,k,x}\circ[\gamma_{x}]^{m_{h,k,x}}\circ\a''_{h,k,x}\thinspace,
\end{equation}
where $m_{h,k,x}$ is an integer which depends only $h$, $k$ and $x$, and the word
length of elements $\a'_{h,k,x}\thinspace, \a''_{h,k,x}$ in $\pi_1(\B\Sigma_g, z)$ is
bounded by some constant $C_{h,x}$ which is independent of $k$.

Denote by $\mathcal{MCG}_g^1$ the mapping class group of the surface $\B\Sigma_g$
with one puncture $z$. Recall that there is the following short exact sequence due to Birman \cite{Birman}
\begin{equation}\label{eq:Birman-SES}
1\to \pi_1(\B\Sigma_g, z)\to\mathcal{MCG}_g^1\to\mathcal{MCG}_g\to 1,
\end{equation}
where $\mathcal{MCG}_g$ is the mapping class group of the surface $\B\Sigma_g$.
Hence we view $\pi_1(\B\Sigma_g, z)$ as a normal subgroup of $\mathcal{MCG}_g^1$.

\begin{prop}\label{P:finite-conj-classes}
Let $g>1$. There exists a finite set $S_g$ of elements in $\mathcal{MCG}_g^1$,
such that for every Morse function $H\colon\B\Sigma_g\to\B R$ and every $x\in\mathrm{Reg}_H$
the loop $[\g_x]\in\pi_1(\B\Sigma_g, z)<\mathcal{MCG}_g^1$
is conjugate to some element in $S_g$.
\end{prop}

\begin{proof} Let $x\in\mathrm{Reg}_H$. If the loop $\g_x(t)$ is homotopically trivial in $\B\Sigma_g$,
then $[\g_x]=1_{\mathcal{MCG}_g^1}$. Suppose that $\g_x(t)$ is homotopically non-trivial in $\B\Sigma_g$.
We say that simple closed curves $\delta,\delta'\in\B\Sigma_g$ are equivalent $\delta\cong\delta'$,
if there exists a homeomorphism $f\colon\B\Sigma_g\to\B\Sigma_g$ such that $f(\delta)=\delta'$.
It follows from classification of surfaces that the set of equivalence classes $\mathcal{E}_g$ is finite.
Let $c_x$ be the simple closed curve defined in \eqref{eq:gamma-reg}. Since $\B\Sigma_g$ and $c_x$ are oriented,
the curve $c_x$ splits in $\B\Sigma_g\setminus\{x\}$ into two simple closed curves $c_{x,+}$ and $c_{x,-}$
which are homotopic in $\B\Sigma_g$, i.e. $c_{x,+}$ and $c_{x,-}$ are boundary curves
of a tubular neighborhood of the curve $c_x$, see Figure \ref{fig:path-split}.
\begin{figure}[htb]
\centerline{\includegraphics[height=1.3in]{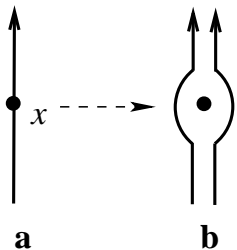}}
\caption{\label{fig:path-split} Part of the curve $c_x$ is shown in Figure \textbf{a}.
Its splitting into curves $c_{x,+}$ and $c_{x,-}$ is shown in Figure \textbf{b}.
The left curve in Figure \textbf{b} is $c_{x,+}$ and the right curve is $c_{x,-}.$}
\end{figure}

The image of the element $[\g_x]$ in $\mathcal{MCG}_g^1$, under the Birman
embedding \eqref{eq:Birman-SES} of $\pi_1(\B\Sigma_g,z)$ into $\mathcal{MCG}_g^1$,
is conjugate to $t_{c_{x,+}}\circ t^{-1}_{c_{x,-}}$, where $t_{c_{x,+}}$ and $t_{c_{x,-}}$
are Dehn twists in $\B\Sigma_g\setminus\{x\}$ about curves $c_{x,+}$ and $c_{x,-}$ respectively, see e.g. \cite[Fact 4.7]{FarbMargalit}.

Note that if $c_x\cong\delta$ then there exists a homeomorphism $f\colon\B\Sigma_g\to\B\Sigma_g$
such that $f(c_x)=\delta$, hence $f(c_{x,+})=\delta_+$ and $f(c_{x,-})=\delta_-$. We have
$$
t_{\delta_+}=f\circ t_{c_{x,+}}\circ f^{-1}\qquad t_{\delta_-}=f\circ t_{c_{x,-}}\circ f^{-1}.
$$
This yields
$$
f\circ(t_{c_{x,+}}\circ t^{-1}_{c_{x,-}})\circ f^{-1}=t_{\delta_+}\circ t^{-1}_{\delta_-}.
$$
In other words, an element $[\g_x]$ is conjugate in $\mathcal{MCG}_g^1$ to
some $t_{\delta_+}\circ t^{-1}_{\delta_-}$, where $\delta$ is a representative of
an equivalence class in $\mathcal{E}_g$. Let $\{\delta_i\}_{i=1}^{\#\mathcal{E}_g}$
be a set of simple closed curves in $\B\Sigma_g$, such that each equivalence class
in $\mathcal{E}_g$ is represented by some $\delta_i$. Let
\begin{equation}\label{eq:finite-set-S}
S_g:=\{t_{\delta_{1,+}}\circ t^{-1}_{\delta_{1,-}},\ldots,t_{\delta_{\#\mathcal{E}_g,+}}\circ t^{-1}_{\delta_{\#\mathcal{E}_g,-}}\}.
\end{equation}
It follows that $[\g_x]$ is conjugate to some element in $S_g$.
Noting that the set $S_g$ depends neither on $H$ nor on $x$, we conclude the proof of the proposition.
\end{proof}

\subsubsection{Mapping class group considerations}

Recall that the group $\pi_1(\B\Sigma_g)$ is a normal subgroup of $\mathcal{MCG}_g^1$.
Denote by $Q_{\mathcal{MCG}_g^1}(\pi_1(\B\Sigma_g), S_g)$ the space of
homogeneous quasi-morphisms on $\pi_1(\B\Sigma_g)$ so that:
\begin{itemize}
\item
For each $\overline{\phi}\in Q_{\mathcal{MCG}_g^1}(\pi_1(\B\Sigma_g), S_g)$
there exists $\widehat{\phi}\in Q(\mathcal{MCG}_g^1)$ such that $\widehat{\phi}|_{\pi_1(\B\Sigma_g)}=\overline{\phi}$,
\item
each $\overline{\phi}$ vanishes on the finite set $S_g$,
\end{itemize}
where $S_g$ is the set defined in \eqref{eq:finite-set-S}. The group $\pi_1(\B\Sigma_g)$
contains a non-abelian free group, and thus is not virtually abelian. It is an infinite
normal subgroup of $\mathcal{MCG}_g^1$ and hence is a non-reducible subgroup of $\mathcal{MCG}_g^1$, see
\cite[Corollary 7.13]{Ivanov}. Now, by a result of Bestvina-Fujiwara
\cite[Theorem 12]{BestvinaFujiwara} we have the following

\begin{cor}\label{cor:inf-dim-qm-restriction}
The space $Q_{\mathcal{MCG}_g^1}(\pi_1(\B\Sigma_g), S_g)$ is infinite dimensional.
\end{cor}

\subsubsection{End of the proof}

Denote by $\mathfrak{Polt}_g\colon Q(\pi_1(\B \Sigma_g))\to Q(\G)$
the map induced by the Polterovich construction.
Recall that $\mathfrak{Polt}_g$ is injective modulo homomorphisms, see Section \ref{ssec-Polterovich-construction}.
Since the space $Q_{\mathcal{MCG}_g^1}(\pi_1(\B\Sigma_g), S_g)$ contains no non-trivial homomorphisms, it follows that the restricted map
$$
\mathfrak{Polt}_g\colon Q_{\mathcal{MCG}_g^1}(\pi_1(\B\Sigma_g), S_g)\hookrightarrow Q(\G)
$$
is injective. Recall that $Q(\G,\A)$ denotes the space of
quasimorphisms on the group $\G=\Ham(\B \Sigma_g)$
that vanish on the set $\A$ of autonomous diffeomorphisms.
Since by Corollary \ref{cor:inf-dim-qm-restriction}
$$
\dim(Q_{\mathcal{MCG}_g^1}(\pi_1(\B\Sigma_g), S_g))=\infty,
$$
and by \cite[Theorem 1]{BrandenburskyLpMetrics} every quasimorphism
which lies in the image of the map $\mathfrak{Polt}_g$ is Lipschitz
with respect to the $L^p$-metric, finishing the proof of
the theorem reduces to proving the following.

\begin{prop}\label{P:key-proposition}
The image of the map
$$
\mathfrak{Polt}_g\colon Q_{\mathcal{MCG}_g^1}(\pi_1(\B\Sigma_g), S_g)\hookrightarrow Q(\G)
$$
lies in the linear space $Q(\G,\A)$.
\end{prop}

\begin{proof}
Let $\overline{\phi}\in Q_{\mathcal{MCG}_g^1}(\pi_1(\B\Sigma_g), S_g)$ and $h\in\G$
an autonomous diffeomorphism. We need to show that $\overline{\Phi}(h)=0$,
where $\overline{\Phi}=\mathfrak{Polt}_g(\overline{\phi})$. Since Morse
functions on $\B\Sigma_g$ form a dense subset in the set of all smooth functions
in $C^1$-topology \cite{Milnor}, by Theorem \ref{T:Morse-Ham-hyperbolic}
it is enough to show that $\overline{\Phi}(h)=0$, where $h$ is the time-one map
of the flow generated by some Morse function $H\colon\B\Sigma_g\to\B R$.
Recall that we have
$$
\overline{\Phi}(h)=\int\limits_{\B\Sigma_g} \lim_{k\to\infty}\frac{\overline{\phi}([h^k_x])}{k}\thinspace\o\thinspace.
$$
Since the set $\mathrm{Reg}_H$ is of full measure in $\B\Sigma_g$,
it is enough to show that for each $x\in \mathrm{Reg}_H$ the following equality holds
$$
\lim_{k\to\infty}\frac{|\phi([h^k_x])|}{k}=0.
$$

The group $\pi_1(\B\Sigma_g)$ admits the following presentation
\begin{equation}\label{eq:fund-gp-presentation}
\pi_1(\B\Sigma_g)=\langle\a_i,\b_i|\hspace{2mm} 1\leq i\leq g,\thinspace \prod_{i=1}^g [\a_i,\b_i]=1\rangle .
\end{equation}
For every $\a\in\pi_1(\B\Sigma_g)$ denote by $l(\a)$ the word length of $\a$ with
respect to the set of generators given in \eqref{eq:fund-gp-presentation}.
Since $\overline{\phi}\in Q_{\mathcal{MCG}_g^1}(\pi_1(\B\Sigma_g), S_g)$, for every $\a\in\pi_1(\B\Sigma_g)$ we have $|\overline{\phi}(\a)|\leq D_{\overline{\phi}}\thinspace l(\a)$.
It follows from \eqref{eq:braid-morse-aut} that for every $k\in\B N$ and $x\in\mathrm{Reg}_H$ we have
$$
[h^k_x]=\a'_{h,k,x}\circ[\gamma_{x}]^{m_{h,k,x}}\circ\a''_{h,k,x}\thinspace,
$$
where $m_{h,k,x}$ is an integer which depends only $h$, $k$ and $x$,
and $l(\a'_{h,k,x})$, $l(\a''_{h,k,x})$ are bounded by some
constant $C_{h,x}>0$  independent of $k$.

Hence for every $k\in\B N$ and $x\in\mathrm{Reg}_H$ we have
$$
0\leq\frac{|\overline{\phi}([h^k_x])|}{k}\leq\frac{|\overline{\phi}(\a'_{h,k,x})|+
|m_{h,k,x}||\overline{\phi}([\gamma_{x}])|+|\overline{\phi}(\a''_{h,k,x})|+2D_{\overline{\phi}}}{k}\hspace{2mm}.
$$
Since $\overline{\phi}\in Q_{\mathcal{MCG}_g^1}(\pi_1(\B\Sigma_g), S_g)$,
by definition $\overline{\phi}$ extends to a homogeneous quasi-morphism on $\mathcal{MCG}_g^1$
and vanishes on the set $S_g$. Hence by Proposition \ref{P:finite-conj-classes}
we have $\overline{\phi}([\gamma_{x}])=0$, and hence
$$
0\leq\frac{|\overline{\phi}([h^k_x])|}{k}\leq\frac{2C_{h,x}\cdot D_{\overline{\phi}}+2D_{\overline{\phi}}}{k}=
\frac{2D_{\overline{\phi}}(C_{h,x}+1)}{k}\hspace{2mm}.
$$
By taking $k\to\infty$ we conclude the proof of the proposition.
\end{proof}

\bibliographystyle{amsplain}
\bibliography{AutonomousMetricsSurfacesRefs}

\end{document}